\documentclass[leqno,12pt]{article} 
\usepackage{amsmath, amssymb}
\usepackage{amsthm} 
\theoremstyle{plain} 
\newtheorem{theorem}{\indent\sc Theorem}[section]
\newtheorem{lemma}[theorem]{\indent\sc Lemma}

\newtheorem{proposition}[theorem]{\indent\sc Proposition}

\theoremstyle{definition} 
\newtheorem{definition}[theorem]{\indent\sc Definition}
\newtheorem{remark}[theorem]{\indent\sc Remark}


\title{Remarks on metallic warped product manifolds} 
\author{Adara M. Blaga and Cristina E. Hre\c tcanu}
\date{}

\begin{document}
\maketitle

\markboth{{\small\it {\hspace{4cm} Metallic warped product manifolds}}}{\small\it{Metallic warped product manifolds
\hspace{4cm}}}


\begin{abstract}
We characterize the metallic structure on the product of two metallic manifolds in terms of metallic maps
and provide a necessary and sufficient condition for the warped product of two locally metallic Riemannian manifolds to be locally metallic.
The particular case of product manifolds is discussed and an example of
metallic warped product Riemannian manifold is provided.
\end{abstract}

\section{Introduction}

Starting from a polynomial structure, which was generally defined by S. I. Goldberg, K. Yano and N. C. Petridis in (\cite{Goldberg1} and \cite{Goldberg2}), we consider a polynomial structure on an $m$-dimensional Riemannian manifold $(M,g)$, called by us a
\textit{metallic structure} (\cite{CrHr}, \cite{Hr2}, \cite{CrHrMu} and \cite{Hr3}), determined by a $(1,1)$-tensor field $J$ which satisfies the equation:
 \begin{equation}
 J^{2}= pJ + qI,
 \end{equation}
where $I$ is the identity operator on the Lie algebra of
vector fields on $M$ identified with the set of smooth sections $\Gamma(T(M))$ (and we'll simply denote $X\in T(M)$), with $p$ and $q$ are non zero natural numbers. From the definition, we easily get the recurrence relation:
\begin{equation}
J^{n+1}=g_{n+1}\cdot J+g_{n}\cdot I,
\end{equation}
where $(\{g_n\}_{n\in\mathbb{N}^*})$ is the generalized secondary Fibonacci sequence defined by $g_{n+1}=pg_n+qg_{n-1}$, $n\geq 1$ with $g_0=0$, $g_1=1$ and $p$, $q\in \mathbb{N}^*$.

If $(M,g)$ is a Riemannian manifold endowed with a metallic structure $J$ such that the Riemannian metric $g$ is $J$-compatible (i.e. $g(J X,Y)=g(X,J Y)$, for any $X, Y\in T(M)$), then $(M,g,J)$ is called a {\it metallic Riemannian manifold}. In this case:
\begin{equation}
g(JX, JY)=pg(X, JY)+qg(X, Y),
\end{equation}
for any $X, Y\in T(M)$.

\pagebreak

It is known (\cite{Hr4}) that an almost product structure $F$ on $M$ induces two metallic structures:
\begin{equation}\label{p2}
J_{\pm}=\pm \frac{2\sigma _{p, q}-p}{2}F+\frac{p}{2}I
\end{equation}
and, conversely, every metallic structure $J$ on $M$ induces two almost product structures:
\begin{equation}
F_{\pm }=\pm \frac{2}{2\sigma _{p, q}-p}J-\frac{p}{2\sigma _{p, q}-p}I,
\end{equation}
where $\sigma _{p,q}=\frac{p+\sqrt{p^{2}+4q}}{2}$ is the metallic number, which is the positive solution of the equation $x^{2}-px-q=0$, for $p$ and $q$ non zero natural numbers.

In particular, if the almost product structure $F$ is compatible with the Riemannian metric, then $J_{+}$ and $J_{-}$ are metallic Riemannian structures.

On a metallic manifold $(M,J)$ there exist two complementary distributions $\mathcal{D}_{l}$ and $\mathcal{D}_{m}$
corresponding to the projection operators $l$ and $m$ (\cite{Hr4}) given by:
\begin{equation}\label{p1}
l=-\frac{1}{2\sigma _{p, q}-p} J+\frac{\sigma _{p, q}}{2\sigma _{p, q}-p} I, \quad m=\frac{1}{2\sigma _{p, q}-p} J+\frac{\sigma _{p, q}-p}{2\sigma _{p, q}-p} I.
\end{equation}

The analogue concept of locally product manifold is considered in the context of metallic geometry. Precisely, we say that the metallic Riemannian manifold $(M,g, J)$ is \textit{locally metallic} if $J$ is parallel with respect to the Levi-Civita connection associated to $g$.

\section{Metallic warped product Riemannian manifolds}

\subsection{Warped product manifolds}

Let $(M_1,g_1)$ and $({M_2},g_2)$ be two Riemannian manifolds of dimensions $n$ and $m$, respectively. Denote by $p_1$ and $p_2$ the projection maps from the product manifold ${M_1}\times {M_2}$ onto ${M_1}$ and ${M_2}$ and by $\widetilde{\varphi}:=\varphi \circ p_1$ the lift to ${M_1}\times {M_2}$ of a smooth function $\varphi$ on ${M_1}$. In this case, we call ${M_1}$ \textit{the base} and ${M_2}$ \textit{the fiber} of ${M_1}\times {M_2}$. The unique element $\widetilde{X}$ of $T({M_1}\times {M_2})$ that is $p_1$-related to $X\in T({M_1})$ and to the zero vector field on ${M_2}$ will be called the \textit{horizontal lift of $X$} and the unique element $\widetilde{V}$ of $T({M_1}\times {M_2})$ that is $p_2$-related to $V\in T({M_2})$ and to the zero vector field on ${M_1}$ will be called the \textit{vertical lift of $V$}.
Also denote by $\mathcal{L}({M_1})$ the set of all horizontal lifts of vector fields on ${M_1}$ and by $\mathcal{L}({M_2})$ the set of all vertical lifts of vector fields on ${M_2}$.

\medskip

For $f>0$ a smooth function on ${M_1}$, consider the Riemannian metric on ${M_1}\times {M_2}$:
\begin{equation}\label{e7}
\widetilde{g}:=p_1^* g_1+(f\circ p_1)^2p_2^*g_2.
\end{equation}

\begin{definition} (\cite{bi})
The product manifold of ${M_1}$ and ${M_2}$ together with the Riemannian metric $\widetilde{g}$ defined by (\ref{e7}) is called \textit{the warped product} of ${M_1}$ and ${M_2}$ by the warping function $f$ [and it is denoted by $(\widetilde{M}:={M_1}\times_f {M_2},\widetilde{g})$].
\end{definition}

Remark that if $f$ is constant (equal to $1$), the warped product becomes the usual product of the Riemannian manifolds.

\bigskip

For $(x, y)\in \widetilde{M}$, we shall identify $X\in T(M_{1})$ with $(X_{x}, 0_{y})\in T_{(x,y)}(\widetilde{M})$ and $Y\in T(M_{2})$ with $(0_{x}, Y_{y}) \in T_{(x,y)}(\widetilde{M})$ (\cite{Baik}).

The projection mappings of $T(M_1 \times M_ 2)$ onto $T(M_1)$ and $T(M_ 2)$, respectively, denoted by $\pi_{1}=:Tp_1$ and $\pi_{2}=:Tp_2$ verify:
\begin{equation}
\pi _{1}+\pi _{2}=I, \quad \pi _{1}^{2}=\pi _{1}, \quad \pi _{2}^{2}=\pi _{2}, \quad \pi _{1} \circ \pi _{2}=\pi _{2} \circ \pi _{1} = 0.
\end{equation}

The Riemannian metric of the warped product manifold $\widetilde{M} =M_ 1 \times_f M_ 2$ equals to:
\begin{equation}
\widetilde{g}(\widetilde{X}, \widetilde{Y})=g_{1}(X_ 1, Y_ 1)+(f\circ p_1)^2g_{2}(X_ 2, Y_ 2),
\end{equation}
for any $\widetilde{X}=(X_ 1,X_ 2), \widetilde{Y}=(Y_ 1,Y_ 2) \in T (\widetilde{M})=T(M_ 1 \times_f M_ 2)$ and we notice that the leaves $M_ 1\times \{y\}$, for $y\in M_2$, are totally geodesic submanifolds of $(\widetilde{M} = M_ 1 \times_f M_ 2, \widetilde{g})$.

\bigskip

If we denote by $\widetilde{\nabla}$, $^{M_1}\nabla$, $^{M_2}\nabla$ the Levi-Civita connections on $\widetilde{M}$, ${M_1}$ and ${M_2}$, we know that for any $X_ 1,Y_ 1 \in T(M_ 1)$ and $X_ 2,Y_ 2 \in T(M_ 2)$ (\cite{lu}):
\begin{equation}\label{e1}
\widetilde{\nabla}_{(X_1,X_2)}(Y_1,Y_2)=(^{M_1}\nabla_{X_1}Y_1-\frac{1}{2}g_2(X_2,Y_2)\cdot grad(f^2),
^{M_2}\nabla_{X_2}Y_2+\frac{1}{2f^2}X_1(f^2)Y_2
+\frac{1}{2f^2}Y_1(f^2)X_2).
\end{equation}

In particular:
$$
\widetilde{\nabla}_{(X,0)}(0,Y)=\widetilde{\nabla}_{(0,Y)}(X,0)=(0,
X(\ln(f))Y).$$

\bigskip

Let $R$, $R_{M_1}$, $R_{M_2}$ be the Riemannian curvature tensors on $\widetilde{M}$, ${M_1}$ and ${M_2}$ and $\widetilde{R_{M_1}}$, $\widetilde{R_{M_2}}$ the lift on $\widetilde{M}$ of $R_{M_1}$ and $R_{M_2}$. Then:
\begin{lemma} (\cite{bi}) \label{l2}
If $(\widetilde{M}:={M_1}\times_f {M_2},\widetilde{g})$ is the warped product of ${M_1}$ and ${M_2}$ by the warping function $f$ and $m>1$, then for any $X$, $Y$, $Z\in \mathcal{L}({M_1})$ and any $U$, $V$, $W\in \mathcal{L}({M_2})$, we have:
\begin{enumerate}
  \item $R(X,Y)Z=\widetilde{R_{M_1}}(X,Y)Z$;
  \item $R(U,X)Y=\frac{1}{f}H^f(X,Y)U$, where $H^f$ is the lift on $\widetilde{M}$ of $Hess(f)$;
  \item $R(X,Y)U=R(U,V)X=0$;
  \item $R(U,V)W=\widetilde{R_{M_2}}(U,V)W-\frac{|grad(f)|^2}{f^2}[g(U,W)V-g(V,W)U]$;
  \item $R(X,U)V=\frac{1}{f}g(U,V)\widetilde{\nabla}_X grad(f)$.
\end{enumerate}
\end{lemma}

Let $S$, $S_{M_1}$, $S_{M_2}$ be the Ricci curvature tensors on $\widetilde{M}$, ${M_1}$ and ${M_2}$ and $\widetilde{S_{M_1}}$, $\widetilde{S_{M_2}}$ the lift on $\widetilde{M}$ of $S_{M_1}$ and $S_{M_2}$. Then:
\begin{lemma} (\cite{bi}) \label{l1}
If $(\widetilde{M}:={M_1}\times_f {M_2},\widetilde{g})$ is the warped product of ${M_1}$ and ${M_2}$ by the warping function $f$ and $m>1$, then for any $X$, $Y\in \mathcal{L}({M_1})$ and any $V$, $W\in \mathcal{L}({M_2})$, we have:
\begin{enumerate}
  \item $S(X,Y)=\widetilde{S_{M_1}}(X,Y)-\frac{m}{f}H^f(X,Y)$, where $H^f$ is the lift on $\widetilde{M}$ of $Hess(f)$;
  \item $S(X,V)=0$;
  \item $S(V,W)=\widetilde{S_{M_2}}(V,W)-\left[\frac{\Delta(f)}{f}+(m-1)\frac{|grad(f)|^2}{f^2}\right]g(V,W)$.
\end{enumerate}
\end{lemma}

\begin{remark}
For the case of product Riemannian manifolds:

i) the Riemannian curvature tensors verify (\cite{Atceken2}):
\begin{equation}\label{e2}
R(\widetilde{X},\widetilde{Y})\widetilde{Z}=(R_{1}(X_{1},Y_{1})Z_{1},R_{2}(X_{2},Y_{2})Z_{2}),
\end{equation}
for any $\widetilde{X}=(X_{1},X_{2}),\widetilde{Y}=(Y_{1},Y_{2}),\widetilde{Z}=(Z_{1},Z_{2})\in T(M_{1} \times M_{2})$, where
$R$, $R_{1}$ and $R_{2}$ are respectively the Riemannian curvature tensors of the Riemannian manifolds $(M_1\times M_2, \widetilde{g})$, $(M_{1},g_{1})$ and $(M_{2},g_{2})$;

ii) the Ricci curvature tensors verify (\cite{Atceken2}):
\begin{equation}
S(\widetilde{X},\widetilde{Y})=S_{1}(X_{1},Y_{1})+S_{2}(X_{2},Y_{2}),
\end{equation}
for any $\widetilde{X}=(X_{1},X_{2}),\widetilde{Y}=(Y_{1},Y_{2}) \in T(M_{1} \times M_{2})$, where
$S$, $S_{1}$ and $S_{2}$ are respectively the Ricci curvature tensors of the Riemannian manifolds $(M_{1} \times M_{2},\widetilde{g})$, $(M_{1},g_{1})$ and $(M_{2},g_{2})$.
\end{remark}

Remark that the Riemannian curvature tensor of a locally metallic Riemannian manifold has the following properties:

\begin{proposition}\label{p}
If $(M,g,J)$ is a locally metallic Riemannian manifold, then for any $X,Y,Z \in T(M)$:
\begin{equation}\label{e3}
R(X,Y)JZ=J(R(X,Y)Z),
\end{equation}
\begin{equation}\label{e4}
R(JX,Y)=R(X,JY),
\end{equation}
\begin{equation}\label{e5}
R(JX,JY)=qR(JX,Y)+pR(X,Y),
\end{equation}
\begin{equation}\label{e6}
R(J^{n+1}X,Y)=g_{n+1}\cdot R(JX,Y)+g_{n}\cdot R(X,Y),
\end{equation}
where $(\{g_n\}_{n\in\mathbb{N}^*})$ is the generalized secondary Fibonacci sequence defined by $g_{n+1}=pg_n+qg_{n-1}$, $n\geq 1$ with $g_0=0$, $g_1=1$ and $p$, $q\in \mathbb{N}^*$.
\end{proposition}
\begin{proof}
The locally metallic condition $\nabla J=0$ is equivalent to $\nabla_XJY=J(\nabla_XY)$, for any $X,Y \in T(M)$ and (\ref{e3}) follows from the definition of $R$.
The relations (\ref{e4}), (\ref{e5}) and (\ref{e6}) follows from the symmetries of $R$ and from the recurrence relation $J^{n+1}=g_{n+1}\cdot J+g_{n}\cdot I$.
\end{proof}

\begin{theorem}
If $(\widetilde{M}:={M_1}\times_f {M_2},\widetilde{g}, \widetilde{J})$ is a locally metallic Riemannian warped product manifold, then $M_2$ is $\widetilde{J}$-invariant submanifold of $\widetilde{M}$.
\end{theorem}
\begin{proof}
Applying (\ref{e4}) from Proposition \ref{p} and Lemma \ref{l2}, we obtain $H^f(X,Y)\widetilde{J}U=H^f(\widetilde{J}X,Y)U$, for any $X$, $Y\in \mathcal{L}({M_1})$ and any $U\in \mathcal{L}({M_2})$, where $H^f$ is the lift on $\widetilde{M}$ of $Hess(f)$.
\end{proof}

\subsection{Metallic warped product Riemannian manifolds}

\textbf{2.2.1. Metallic Riemannian structure on $(\widetilde{M},\widetilde{g})$ induced by the projection operators}\\

The endomorphism
\begin{equation}
F:=\pi _{1}-\pi _{2}
\end{equation}
verifies $F^{2}=I$ and $\widetilde{g}(F\widetilde{X}, \widetilde{Y})= \widetilde{g}(\widetilde{X}, F\widetilde{Y})$, thus $F$ is an almost product structure on $M_ 1 \times M_ 2$.

By using relations (\ref{p2}) we can construct on $M_ 1 \times M_ 2$ two metallic structures, given by:
\begin{equation}
\widetilde{J}_{\pm}=\pm \frac{2\sigma_{p,q}-p}{2} F+\frac{p}{2}I.
\end{equation}

Also from $\widetilde{g}(F\widetilde{X}, \widetilde{Y})= \widetilde{g}(\widetilde{X}, F\widetilde{Y})$ follows $\widetilde{g}(\widetilde{J}_{\pm}\widetilde{X}, \widetilde{Y})= \widetilde{g}(\widetilde{X}, \widetilde{J}_{\pm}\widetilde{Y})$. Therefore, we can state the following result:

\begin{theorem} There exist two metallic Riemannian structure $\widetilde{J}_{\pm}$ on $(\widetilde{M}, \widetilde{g})$ given by:
\begin{equation}\label{e66}
\widetilde{J}_{\pm}=\pm \frac{2\sigma_{p,q}-p}{2} F+\frac{p}{2}I,
\end{equation}
where $\widetilde{M} = M_ 1 \times_f M_ 2$ and $\widetilde{g}(\widetilde{X}, \widetilde{Y})=g_ 1(X_ 1, Y_ 1)+(f\circ p_1)^2g_ 2(X_ 2, Y_ 2)$, for any $\widetilde{X}=(X_ 1,X_ 2), \widetilde{Y}=(Y_1,Y_ 2) \in T (\widetilde{M})=T(M_ 1\times_f M_ 2)$.
\end{theorem}

Remark that for $\widetilde{J}_{+}= \frac{\displaystyle 2\sigma_{p,q}-p}{\displaystyle 2} F+\frac{\displaystyle p}{\displaystyle 2}I$, the projection operators are $\pi _{1}=m, \: \pi _{2}=l$ and for $\widetilde{J}_{-}=- \frac{\displaystyle 2\sigma_{p,q}-p}{\displaystyle 2} F+\frac{\displaystyle p}{\displaystyle 2}I$ we have $\pi _{1}=l, \: \pi _{2}=m$, where $m$ and $l$ are given by (\ref{p1}).

\begin{remark}
If we denote by $\widetilde{\nabla}$ the Levi-Civita connection on $\widetilde{M}$ with respect to $\widetilde{g}$, we can check that $\widetilde{\nabla}F=0$ [hence $\widetilde{\nabla}\widetilde{J}_{\pm}=0$ and so $(\widetilde{M}=M_{1} \times_f M_{2},\widetilde{g},\widetilde{J}_{\pm})$ is a locally metallic Riemannian manifold].
\end{remark}

For the case of a product Riemannian manifold $(\widetilde{M} = M_ 1 \times M_ 2, \widetilde{g})$ with $\widetilde{g}$ given by (\ref{e7}) for $f=1$ and $\widetilde{J}_{\pm}$ defined by (\ref{e66}), we deduce that the Riemann curvature of $\widetilde{\nabla}$ verifies (\ref{e3}), (\ref{e4}), (\ref{e5}), (\ref{e6}).

\bigskip

\textbf{2.2.2. Metallic Riemannian structure on $(\widetilde{M},\widetilde{g})$ induced by two metallic structures on $M_1$ and $M_2$}\\

For any vector field $\widetilde{X}=(X, Y) \in T(M_1\times M_2)$ we define a linear map $\widetilde{J}$ of tangent space $T(M_1\times M_2)$ into itself by:
\begin{equation}
\widetilde{J}\widetilde{X}=(J_{1}X, J_{2}Y),
\end{equation}
where $J_{1}$ and $J_{2}$ are two metallic structures defined on $M_{1}$ and $M_{2}$, respectively, with $J_i^2=pJ_i+qI$, $i\in \{1,2\}$ and $p$, $q$ non zero natural numbers.
It follows that:
\begin{equation}
\widetilde{J}^{2}\widetilde{X} =\widetilde{J}(J_{1}X, J_{2}Y)= (J_{1}^{2}X,J_{2}^{2}Y)=
(pJ_{1}X +qX, pJ_{2}Y +qY)= p(J_{1}X,J_{2}Y) +q(X,Y).
\end{equation}
Also from $g_i(J_iX_i,Y_i)=g_i(X_i,J_iY_i)$, $i\in \{1,2\}$, we get $\widetilde{g}(\widetilde{J}\widetilde{X},\widetilde{Y})=\widetilde{g}(\widetilde{X},\widetilde{J}\widetilde{Y})$. Therefore, we can state the following result:

\begin{theorem} If $(M_ 1, g_ 1, J_ 1)$ and $(M_ 2, g_ 2, J_ 2)$ are metallic Riemannian manifolds with $J_i^2=pJ_i+qI$, $i\in \{1,2\}$ and $p$, $q$ non zero natural numbers, then there exists a metallic Riemannian structure $\widetilde{J}$ on $(\widetilde{M}, \widetilde{g})$ given by:
\begin{equation}\label{e67}
\widetilde{J}\widetilde{X}=(J_{1}X, J_{2}Y),
\end{equation}
for any $\widetilde{X}=(X,Y) \in T (\widetilde{M})$, where $\widetilde{M} = M_ 1 \times_f M_ 2$ and $\widetilde{g}(\widetilde{X}, \widetilde{Y})=g_ 1(X_ 1, Y_ 1)+(f\circ p_1)^2g_ 2(X_ 2, Y_ 2)$, for any $\widetilde{X}=(X_ 1,X_ 2), \widetilde{Y}=(Y_1,Y_ 2) \in T (\widetilde{M})=T(M_ 1\times_f M_ 2)$.
\end{theorem}

For the case of a product Riemannian manifold $(\widetilde{M} = M_ 1 \times M_ 2, \widetilde{g})$ with $\widetilde{g}$ given by (\ref{e7}) for $f=1$ and $\widetilde{J}_{\pm}$ defined by (\ref{e66}), we deduce that the Riemann curvature of $\widetilde{\nabla}$ verifies (\ref{e3}), (\ref{e4}), (\ref{e5}), (\ref{e6}).

\bigskip

Now we shall obtain a characterization of the metallic structure on the product of two metallic manifolds $(M_{1},J_{1})$ and $(M_{2},J_{2})$ in terms of \textit{metallic maps}, that are smooth maps $\Phi:M_1\rightarrow M_2$ satisfying:
$$T\Phi \circ J_1=J_2\circ T\Phi.$$

Similarly like in the case of Golden manifolds (\cite{blhr}), we have:
\begin{proposition}
The metallic structure $\widetilde{J}:=(J_1,J_2)$ given by (\ref{e67}) is the only metallic structure on the product manifold $\widetilde{M}=M_{1} \times M_{2}$ such that the projections $p_1$ and $p_2$ on the two factors $M_1$ and $M_2$ are metallic maps.
\end{proposition}

A necessary and sufficient condition for the warped product of two locally metallic Riemannian manifolds to be locally metallic will be further provided:

\begin{theorem}
Let $(\widetilde{M}=M_{1} \times_f M_{2},\widetilde{g},\widetilde{J})$ (with $\widetilde{g}$ given by (\ref{e7}) and $\widetilde{J}$ given by (\ref{e67})) be the warped product of the locally metallic Riemannian manifolds $(M_{1},g_{1}, J_{1})$ and $(M_{2},g_{2},J_{2})$. Then $(\widetilde{M}=M_{1} \times_f M_{2},\widetilde{g},\widetilde{J})$ is locally metallic if and only if:
$$\left\{
  \begin{array}{ll}
    (df^2 \circ J_1)\otimes I=df^2\otimes J_2 \\
    g_2(J_1\cdot,\cdot)\cdot grad(f^2)=g_2(\cdot,\cdot)\cdot J_1(grad(f^2))
  \end{array}
\right. .$$
\end{theorem}
\begin{proof}
Replacing the expression of $\widetilde{\nabla}$ from (\ref{e1}), under the assumptions $^{M_1}\nabla J_1=0$ and $^{M_2}\nabla J_2=0$ we obtain the conclusion.
\end{proof}

\begin{theorem}
Let $(\widetilde{M}=M_{1} \times_f M_{2},\widetilde{g},\widetilde{J})$ (with $\widetilde{g}$ given by (\ref{e7}) and $\widetilde{J}$ (\ref{e67})) be the warped product of the metallic Riemannian manifolds $(M_{1},g_{1}, J_{1})$ and $(M_{2},g_{2},J_{2})$. If $M_1$ and $M_2$ have $J_1$- and $J_2$-invariant Ricci tensors, respectively (i.e. $Q_{M_i}\circ J_i=J_i\circ Q_{M_i}$, $i\in \{1,2\}$), then $\widetilde{M}$ has $\widetilde{J}$-invariant Ricci tensor if and only if
$$Hess(f)(J_1\cdot,\cdot)-Hess(f)(\cdot,J_1\cdot)\in \{0\}\times T(M_2).$$
\end{theorem}
\begin{proof}
If we denote by $S$, $S_{M_1}$, $S_{M_2}$ the Ricci curvature tensors on $\widetilde{M}$, ${M_1}$ and ${M_2}$ and $\widetilde{S_{M_1}}$, $\widetilde{S_{M_2}}$ the lift on $\widetilde{M}$ of $S_{M_1}$ and $S_{M_2}$,
by using Lemma \ref{l1}, for any
$X$, $Y\in \mathcal{L}({M_1})$, we have:
$$S(\widetilde{J}X,Y)=\widetilde{S_{M_1}}(\widetilde{J}X,Y)-\frac{m}{f}H^f(\widetilde{J}X,Y)=\widetilde{S_{M_1}}(X,\widetilde{J}Y)-\frac{m}{f}H^f(\widetilde{J}X,Y)=$$
$$=S(X,\widetilde{J}Y)+\frac{m}{f}H^f(X,\widetilde{J}Y)-\frac{m}{f}H^f(\widetilde{J}X,Y),$$
where $H^f$ is the lift on $\widetilde{M}$ of $Hess(f)$. Also, for any $V$, $W\in \mathcal{L}({M_2})$, we obtain:
$$S(\widetilde{J}V,W)=\widetilde{S_{M_2}}(\widetilde{J}V,W)-[f\Delta(f)+(m-1)|grad(f)|^2]g_2(J_2V,W)=$$
$$=\widetilde{S_{M_2}}(V,\widetilde{J}W)-[f\Delta(f)+(m-1)|grad(f)|^2]g_2(V,J_2W)=S(V,\widetilde{J}W).$$
\end{proof}

\section{Example}

Let $M:=\{(u,\alpha_{1},\alpha_{2},...,\alpha_{n}), u > 0, \alpha_{i} \in [0, \frac{\pi}{2}], i \in \{1,...,n\} \}$ and $f:M\rightarrow \mathbb{R}^{2n}$ is the immersion given by:
\begin{equation} \label{e74}
f(u,\alpha_{1},...,\alpha_{n}):=(u \cos\alpha_{1},u \sin\alpha_{1},...,u \cos\alpha_{n},u \sin\alpha_{n}).
\end{equation}

We can find a local orthonormal frame of the submanifold $M$ in $\mathbb{R}^{2n}$, spanned by the vectors:
\begin{equation} \label{e75}
Z_{0}=\sum_{i=1}^{n} \left(\cos \alpha_{i}\frac{\partial}{\partial x_{i}}+\sin \alpha_{i}\frac{\partial}{\partial y_{i}}\right), \quad
Z_{i}= -u \sin \alpha_{i}\frac{\partial}{\partial x_{i}}+u \cos \alpha_{i}\frac{\partial}{\partial y_{i}},
\end{equation}
for any $i \in \{1,...,n\}$.

We remark that $\|Z_{0}\|^{2}=n$, $\|Z_{i}\|^{2}=u^{2}$,
 $Z_{0} \bot Z_{i}$, for any $i \in \{1,...,n\}$ and $Z_{i} \bot Z_{j}$, for $i \neq j$, where $i, j \in \{1,...,n\}$.

In the next considerations, we shall denote by: $$(X^{1},Y^{1},...,X^{k},Y^{k},X^{k+1},Y^{k+1},...,X^{n},Y^{n})=:(X^{i},Y^{i},X^{j},Y^{j}),$$ for any $k \in \{2,...,n-1\}$, $ i \in \{1,...,k\}$ and $j \in \{k+1,...,n\}$.

Let $J:\mathbb{R}^{2n}\rightarrow \mathbb{R}^{2n}$ be the $(1,1)$-tensor field defined by:
\begin{equation} \label{e76}
J(X^{i},Y^{i},X^{j},Y^{j}):=(\sigma X^{i},\sigma Y^{i},\overline{\sigma}X^{j},\overline{\sigma}Y^{j}),
\end{equation}
for any $k \in \{2,...,n-1\}$, $ i \in \{1,...,k\}$ and $j \in \{k+1,...,n\}$, where $\sigma:=\sigma_{p,q}$ is the metallic number and $\overline{\sigma}=1-\sigma$. It is easy to verify that $J$ is a metallic structure on $\mathbb{R}^{2n}$ (i.e. $J^{2}=pJ+qI$).

Moreover, the metric $\overline{g}$, given by the scalar product $\langle\cdot,\cdot\rangle$ on $\mathbb{R}^{2n}$, is $J$-compatible and $(\mathbb{R}^{2n},\overline{g},J)$ is a metallic Riemannian manifold.

From (\ref{e75}) we get:
$$
JZ_{0}=\sigma \sum_{i=1}^{k}\left(\cos \alpha_{i}\frac{\partial}{\partial x_{i}}+\sin \alpha_{i}\frac{\partial}{\partial y_{i}}\right)+
\overline{\sigma}\sum_{j=k+1}^{n}\left(\cos \alpha_{j}\frac{\partial}{\partial x_{j}}+\sin \alpha_{j}\frac{\partial}{\partial y_{j}}\right)
$$
and, for any $k \in \{2,...,n-1\}, i \in \{1,...,k\}$ and $ j \in \{k+1,...,n\}$ we get:
$$JZ_{i}=\sigma Z_{i}, \quad JZ_{j}=\overline{\sigma}Z_{j}.$$

We can verify that $JZ_{0}$ is orthogonal to $span\{Z_{1},...,Z_{n}\}$ and
\begin{equation} \label{e77}
\cos(\widehat{JZ_{0},Z_{0}}) = \frac{k\sigma +(n-k)\bar{\sigma}}{\sqrt{n(k\sigma^{2} +(n-k)\bar{\sigma}^{2})}}.
\end{equation}

If we consider the distributions $D_{1}=span \{Z_{1},...,Z_{n}\}$ and $D_{2}=span\{Z_{0}\}$, then $D_{1} \oplus D_{2}$ is a semi-slant submanifold of the metallic Riemannian manifold $(\mathbb{R}^{2n},\langle\cdot,\cdot\rangle,J)$. The Riemannian metric tensor on $D_{1} \oplus D_{2}$ is given by $g=n du^{2} + u^{2}\sum_{i=1}^{n}d\alpha_{i}^{2}$, thus $M$ is a warped product submanifold of the metallic Riemannian manifold $(\mathbb{R}^{2n},\langle\cdot,\cdot\rangle,J)$.

\bibliographystyle{amsplain}

{\bf Adara M. Blaga} \\
West University of Timi\c{s}oara\\
Timi\c{s}oara, 300223, Romania\\
e-mail: adarablaga@yahoo.com

\medskip

{\bf Cristina E. Hre\c{t}canu}\\
Stefan cel Mare University of Suceava \\
Suceava, 720229, Romania\\
e-mail: criselenab@yahoo.com

\end{document}